\documentclass[a4paper]{article}

\usepackage[cyr]{aeguill}
\usepackage[english]{babel}
\usepackage[utf8]{inputenc}
\usepackage[T1]{fontenc}

\usepackage[a4paper,top=3cm,bottom=2cm,left=3cm,right=3cm,marginparwidth=1.75cm]{geometry}

\usepackage{amsmath}
\usepackage{graphicx}
\usepackage[colorinlistoftodos]{todonotes}
\usepackage[colorlinks=true, allcolors=blue]{hyperref}
\usepackage{amssymb}
\usepackage{amsthm}
\usepackage{bbold}
\usepackage[all]{xy}
\usepackage{dirtytalk}
\usepackage[titletoc,title]{appendix}
\usepackage{ stmaryrd }

\newtheorem{theorem}{Theorem}[section]

\newtheorem{lemma}[theorem]{Lemma}
\newtheorem{proposition}[theorem]{Proposition}

\theoremstyle{definition}
\newtheorem{definition}{Definition}[section]
\newtheorem{example}{Example}[section]

\newtheorem{remark}[example]{Remark}

\title{On oscillatory integrals with Hölder phases}
\author{Gaétan Leclerc}
\date{ }

\begin{document}

\maketitle

\begin{abstract}

We exhibit a family of autosimilar Hölder maps that satisfies a \say{fractal} version of the Van Der Corput Lemma, despite not being absolutely continuous. The result is a direct consequence of a recent work of Sahlsten and Steven \cite{SS20}, which is based on a powerful theorem of Bourgain known as a \say{sum-product phenomenon} estimate. We give a substantially simpler proof of this fact in our particular context, using an elementary method inspired from \cite{BD17} to check the \say{non-concentration estimates} that are needed to apply the sum-product phenomenon. This method allows us to gain additional control over the decay rate.

\end{abstract}

\section{About oscillatory integrals}

\subsection{The Van Der Corput Lemma}

It is an understatement to say that oscillatory integrals play a major role in modern analysis. One of the simplest questions we can ask about oscillatory integrals is the following: give some sufficient conditions on the phase function $\psi$ so that so that the associated integral exhibit power decay, in the sense that there exists $\delta > 0$ such that for all large $|\xi|$, we have $$ \left| \int_0^1 e^{i \xi \psi(x)} dx \right| \leq  |\xi|^{-\delta} .$$
The first result of this kind that comes to mind is the useful Van Der Corput lemma, which we recall:

\begin{lemma}[Van Der Corput]

Let $\psi : [0,1] \rightarrow \mathbb{R}$ be a  $C^{k+1}$ phase ($k \geq 2$) such that the $k$-th derivative satisfies $\phi^{(k)} \geq 1$. Then, there exists $C_k>0$ such that, for all $\xi \geq 1$, 
$$ \left|  \int_0^1 e^{i \xi \psi(x)} dx \right| \leq C_k \xi^{-1/k} .$$

\end{lemma}

In substance, Lemma 1.1 states that our oscillatory integral exhibit power decay as soon as our (smooth) phase satisfies a form a non-concentration hypothesis (the condition on the derivative). Variants of this lemma (e.g. the non-stationary phase) also relies on a non-concentration of the phase, and on its smoothness. It is surprising to find that, to the author's knowledge, no deterministic example of (non-absolutely continuous) Hölder maps $\psi$ are known to satisfies this type of result. Yet, non-smooth maps appear regularly, and one must find a way to deal with them: for example, in the context of hyperbolic dynamical systems, the stable/unstable foliation is known to be only Hölder regular. Conjugacy between dynamical systems are often only Hölder, and invariant sets (such as Julia sets in the context of conformal dynamics) are often very non-smooth. \\

The goal of our paper is to construct deterministic examples of Hölder phases satisfying a Van Der Corput type of estimate, despite not being absolutely continuous. The lack of smoothness of the phase will be replaced by a form of autosimiliarity which will play a key role in the proof.

\subsection{A probabilistic example: the Brownian motion}

Before stating our main result, we will discuss some estimates that have been proved in a random setting by Kahane \cite{Ka85}. Let $(X_n)_{n \geq 0}$ and $(Y_n)_{n \geq 1}$ be some i.i.d. random variables following a normalized Gaussian distribution $\mathcal{N}(0,1)$. We define the Brownian motion (or Wiener process) on $[0,1]$ by the following stochastic process:
$$ W(t) := X_0 t + \sqrt{2} \sum_{n=1}^\infty \frac{1}{2 \pi n} \Big( X_n \sin( 2 \pi n t) 
 + Y_n (1- \cos(2 \pi n t))\Big) ,$$
where convergence takes place almost surely in the $L^2(0,1)$ sense. (Indeed, the fact that \\
$\mathbb{E}( \sum_{k \geq 1} \exp(-X_k^2/4) k^{-2})< \infty$ implies that the series inside is finite a.s., and so $X_k = O(\sqrt{\ln \ln k})$.) It is known that, for any $\alpha < 1/2$, $W$ defines almost surely a $\alpha$-Holder map, and for any $\alpha \geq 1/2$ is almost surely not $\alpha$-Holder. (\cite{Ka85}, p. 235, Th. 2 p. 236 and Th. 3 p. 241) The following estimate holds.

\begin{proposition}[\cite{Ka85}, p. 255]

Almost surely, there exists $C>0$ such that for all $|\xi| \geq 1,$ we have $$ \left| \int_0^1 e^{i \xi W(t)} dt \right| \leq C |\xi|^{-1} \sqrt{\ln |\xi|}. $$

\end{proposition}

This proposition is very inspiring: it suggest that a form of Van Der Corput Lemma should hold in some generic sense for some genuinely \say{fractal} phases. More precisely, the lack of regularity of the Brownian motion is compensated by its statistical autosimilarity (\cite{Ka85}, Th. 1 p. 234): it is known that, for any $c>0$, the process
$$ t \mapsto \sqrt{c} \ W(t/c) $$
also defines a Brownian motion. Moreover, for any $a>0$, the following scaled increments all follow the same gaussian law: $$ \frac{W(t+a) - W(t) }{\sqrt{a}}\sim \mathcal{N}(0,1) .$$   Those properties of the Brownian motion tell us that $W(t)$ behaves in the same way in every scale, which allows us to easily \say{zoom in} in the proofs as $\xi$ grows. \\

It is thus natural for us to search for a deterministic canditate in the realm of \say{fractal} functions. In the case of the Brownian motion, the property $ \sqrt{c} \ W(t/c) \sim W(t)$ can be formally rewritten as a form of conjugacy: in a sense, $W$ acts like a conjugacy between $x \mapsto cx$ and $x \mapsto \sqrt{c}x$. This may be a hint for us to consider conjugacies of dynamical systems as good candidates for phases.

\subsection{Our deterministic setting}

Our explicit family of autosimilar phases $\psi:[0,1] \rightarrow \mathbb{R}$ will be constructed as conjugacies between the doubling map and some perturbation. We denote by $\mathbb{S}^1 := \mathbb{R}/\mathbb{Z}$ the circle. Define the doubling map $f_0 : \mathbb{S} \rightarrow \mathbb{S}$ by $f_0(x) := 2x$.
 To state the main result, we need to recall a useful fact on perturbations of expanding maps.

\begin{proposition}[\cite{KH95}, Th.19.1.2 and Th.18.2.1]

Let $0<\alpha<1$. Then there exists $\delta>0$ such that the following holds. Let $f:\mathbb{S}^1 \rightarrow \mathbb{S}^1$ be a $C^{1+\alpha}$ $\delta$-perturbation of the doubling map $f_0: x \in  \mathbb{S}^1 \mapsto 2x \in  \mathbb{S}^1$, meaning that
$$ \|f-f_0\|_{C^{1+\alpha}} < \delta .$$

Then, reducing $\alpha$ if necessary, there exists a $\alpha$-Holder conjugacy $\psi : ( \mathbb{S}^1,f_0) \rightarrow ( \mathbb{S}^1,f)$. In other words, $\psi : \mathbb{S}^1 \rightarrow \mathbb{S}^1$ is a homeomorphism, Hölder with Hölder inverse, and $\psi \circ f_0 = f \circ \psi$. Generically, $\psi$ is not absolutely continuous (meaning that its derivative in the sense of distributions is not in $L^1(\mathbb{S})$). 

\end{proposition}

The fact that $\psi$ is generically not absolutely continuous is because any such conjugacy must be $C^{1+\alpha}$ (see \cite{SS85}), which is not allowed if $f'(\psi(0)) \neq 2$, for example. (The derivative is taken in the sense that $f$ can be identified with an increasing, 1-periodic and $C^{1+\alpha}$ function $\mathbb{R} \rightarrow \mathbb{R}$.) We are ready to state our main theorem.

\begin{theorem}

Let $f_0:\mathbb{S}^1 \longrightarrow \mathbb{S}^1$ be the doubling map. Let $f$ be a $C^{2+\alpha}$ $\delta-$perturbation of $f_0$. Let $\psi : ( \mathbb{S}^1,f_0) \rightarrow ( \mathbb{S}^1,f)$ be the $\alpha$-Hölder conjugacy. Then, there exists $C>0$ and $\rho>0$ such that:

$$ \forall \xi \in \mathbb{R}^*, \ \left| \int_0^1 e^{i \xi \psi(x)} dx \right| \leq C |\xi|^{-\rho} .$$

\end{theorem}

We can re-write Theorem 1.4 in a measure-theoretic form.

\begin{definition}

Denote by $\mu := \psi_* (dx) $ the pushforward of the lebesgue measure by $\psi$, so that
    $$ \int_0^1 e^{i \xi \psi(x)} dx = \int_0^1 e^{i \xi x} d\mu(x)$$
    
is the Fourier transform $\widehat{\mu}(\xi)$ of the measure $\mu$. This probability measure is $f$-invariant on $\mathbb{S}^1$, since the Lebesgue measure is $f_0$-invariant on the circle.

\end{definition}

\begin{remark}
The measure $\mu$ is known as the \emph{measure of maximal entropy} for the dynamical system $(\mathbb{S},f)$, as it is the push-forward via $\psi$ of the Lebesgue measure $\lambda$, which is the measure of maximal entropy for the doubling map. Indeed, the topological entropy $h(f_0)$ of $(\mathbb{S},f_0)$ is $\ln 2$ (see \cite{BS02} section 2.5) and the measure-theoretic entropy $h_\lambda(f_0)$ is also $\ln 2$ (see \cite{BS02}, section 9.4, and  Th. 9.5.4).
\end{remark}

To prove this result, we separate two cases: one where $f$ satisfies a linearity condition, and one where $f$ is \emph{totally non linear.} The \emph{total nonlinearity condition} (TNL) is defined as follows.

\begin{definition}
We say that $f$ satisfies (TNL) if there exists no Lipschitz map $\theta: \mathbb{S}\setminus\{\psi(0),\psi(1/2)\} \rightarrow \mathbb{R}$ and no locally constant map $\kappa: \mathbb{S}\setminus\{\psi(0),\psi(1/2)\} \rightarrow \mathbb{R}$  such that, on $\mathbb{S}\setminus\{\psi(0),\psi(1/2)\}$, $$ \ln f' = \theta \circ f - \theta + \kappa .$$
\end{definition}

\begin{lemma}

If $f$ doesn't satisfies $(TNL)$, then Theorem 1.4 holds.

\end{lemma}

\begin{proof}

Suppose that there exists a Lipschitz map $\theta: \mathbb{S}\setminus\{\psi(0),\psi(1/2)\} \rightarrow \mathbb{R}$ and a locally constant map $\kappa: \mathbb{S}\setminus\{\psi(0),\psi(1/2)\} \rightarrow \mathbb{R}$  such that $$ \ln f' = \theta \circ f - \theta + \kappa .$$
 on $\mathbb{S}\setminus\{\psi(0),\psi(1/2)\}$. Since we are working with the doubling map $f_0$, $\psi(0)$ is a fixed point of $f$ that bounds our two intervals, and this implies that $\kappa$ is constant. Indeed, we see that
 $$ \ln f'(\psi(0^+)) = \theta( f(\psi(0^+)) ) -  \theta( \psi(0^+) ) + \kappa(\psi(0^+)) =  \kappa(\psi(0^+)) $$
 and, similarly, $\ln f'(\psi(0^-)) = \kappa(\psi(0^-))$. Hence $ \ln f' = \theta \circ f - \theta + \kappa $ for some constant $\kappa$. We say that $f$ is \emph{cohomologous} to a constant. It is then well known (\cite{PP90}, Th 3.6, and \cite{Ba18}, Th 2.2) that $\mu$, the measure of maximal entropy for $(\mathbb{S},f)$, is equal to the \emph{SRB measure}, that is, the only invariant probability measure that is absolutely continuous with respect to the Lebesgue measure. Moreover, in our case, the density of the $SRB$ measure is $C^1$, which is enough to ensure power decay for $\widehat{\mu}$. \end{proof}

The totally nonlinear case directly follows from the work of Sahlsten and Steven \cite{SS20}, dealing with the power decay of the Fourier transform of equilibrium states for one dimensional expanding maps. The technique used is a generalization of some previous work of Bourgain-Dyatlov \cite{BD17}, Li-Naud-Pan \cite{LNP19}, and also Jordan-Sahlsten \cite{JS16}. The proof is fairly technical in the general case, and our goal is to provide a self-contained and elementary proof for an explicit family of perturbations. More precisely, for $x \in [0,1)$, let $\Phi(x) := (x - 1/4) \mathbb{1}_{[0,1/2[}(x) + (3/4-x) \mathbb{1}_{[1/2,1[}(x) $. Extend $\Phi$ to a 1-periodic function, and then set, for $\delta>0$ small enough:
$$ f_\delta(x) := z_\delta \int_0^{2x} e^{\delta \Phi(t)} dt $$
where $\delta>0$ is a normalization factor chosen so that $f_\delta(1/2)=1$. This maps factors into a perturbation of the doubling map on the circle. Denote the associated conjugacy $\Psi_\delta$. We prove the following.

\begin{theorem}
There exists $C > 0$ and $\rho>0$ such that, if $\delta>0$ is small enough:

$$\forall |\xi| \geq 1, \ \left| \int_0^1 e^{i \xi \psi_\delta(x)} dx \right| \leq C \delta^{-1} |\xi|^{- \rho} .$$

\end{theorem}

It is worth noticing that our explicit approach allows to see that the exponent $\rho$ doesn't depends on $\delta$: this answers, in our particular case, a question found in \cite{SS20} about the dependence of this exponent on the dynamics. \\

The plan of the paper is the following. Section 2, 3 and 4 shows how one can reduce Theorem 1.6 to checking a \say{nonlinearity estimate} on $f$ for a general $C^{1+\text{Lip}}$ ($C^1$ with Lipschitz derivative) perturbation of the doubling map. In section 5, we see how one can check those estimates for the explicit perturbation defined above in an elementary way, inspired from \cite{BD17}. (In general, this last step uses some additional technology that requires us to work with $C^{2+\alpha}$ perturbations.)

\begin{remark}
It was pointed out to me by Frederic Naud that another example of Holder phase satisfying this \say{fractal Van der corput lemma} can be easily constructed. Consider the doubling map, but this time, seen as the map $z \mapsto z^2$ restricted to the unit circle $\mathbb{U} \subset \mathbb{C}$. If $c$ is a small enough complex number, then there exists a topological circle $J_c$ on which the dynamical system $z \mapsto z^2+c$ is well defined (\cite{Ly86}, section 1.16). The two dynamics are then conjugated by a quasiconformal mapping $\psi : \mathbb{U} \rightarrow J_c$ (which is Holder, \cite{FS58}). It follows from the author previous work $\cite{Le21}$ that there exists $\varepsilon>0$ and $C>0$ such that
$$ \forall z \in \mathbb{C},\ \left| \int_\mathbb{U} e^{i \text{Re}\left(z \psi(\theta)\right) } d\theta \right| \leq C (1+|z|)^{-\varepsilon} .$$

This remark was the motivation for this paper.

\end{remark}

\section{Preliminary facts}

We choose a $C^{1+\text{Lip}}$ $\delta$-perturbation $f$ of the doubling map $f_0 : x \mapsto 2x \mod 1$, for $\delta$ small enough. Denote the Holder conjugacy $\psi$. First of all, we define some inverse branches for the doubling map.

\begin{definition}

Define $S_0^{(0)} := [0,1/2)$ and $S_1^{(0)} := [1/2,1)$. This is a partition of $[0,1) \simeq \mathbb{S}^1 $ adapted to the doubling map. Define the associated inverse branches by:

$$ \begin{array}[t]{lrcl}
 g_0^{(0)}: & [0,1) & \longrightarrow & S_0^{(0)} \\
    & x  & \longmapsto &  x/2 \end{array}  \ , \quad \begin{array}[t]{lrcl}
 g_1^{(0)}: &  [0,1)  & \longrightarrow &  S_1^{(0)} \\
    & x  & \longmapsto &  (x+1)/2 \end{array}  $$
    
For any finite word $\mathbf{a} = a_1 \dots a_n \in \{0,1\}^n$, define 
$$ g_{\mathbf{a}}^{(0)} := g_{a_1}^{(0)} \dots g_{a_n}^{(0)} .$$
The cylinder set associated to the word $\mathbf{a}$ is defined by $ S_\mathbf{a}^{(0)} := g_{\mathbf{a}}^{(0)}(\mathbb{S}) $, and the collection  \\ $\{S_\mathbf{a}^{(0)} \ , \ \mathbf{a} \in \{0,1\}^n \}$ is a partition of $\mathbb{S}^1$.

\end{definition}

Recall that the Lebesgue measure is invariant by the doubling map. Moreover, for any measurable map $h : \mathbb{S}^1 \rightarrow \mathbb{C}$, we have the identity

$$ \int_{\mathbb{S}} h(x) dx = \frac{1}{2} \left( \int_{\mathbb{S}} h(g_0^{(0)}(x)) dx + \int_{\mathbb{S}} h(g_1^{(0)}(x)) dx \right) ,$$
which gives, by induction:

$$ \int_{\mathbb{S}} h(x) dx = \frac{1}{2^n} \sum_{\mathbf{a} \in \{0,1\}^n } \int_{\mathbb{S}} h\left( g_{\mathbf{a}}^{(0)}(x)\right) dx .$$

Now, we use our conjugacy $\psi$ to define similar inverse branches and partitions for $f$.

\begin{definition}
For any finite word $\mathbf{a} \in \{0,1\}^n$, define
$$ g_{\mathbf{a}} := \psi \circ g_{\mathbf{a}}^{(0)} \circ \psi^{-1} \text{ and } S_{\mathbf{a}} := \psi\left( S_{\mathbf{a}}^{(0)} \right) .$$
Notice that $S_{\mathbf{a}} = g_{\mathbf{a}}(\mathbb{S})$, and that $g_{\mathbf{a}}$ is a local inverse of $f$. In particular, it is $C^{1+\text{Lip}}$ on $[0,1)$.
\end{definition}

By definition of the inverse branches, the associated partition, and the pushforward measure $\mu$, we see that we have the following identity, holding for measurable maps $h:\mathbb{S} \rightarrow \mathbb{C}$:

$$ \int_\mathbb{S} h d\mu = \frac{1}{2^n} \sum_{\mathbf{a} \in \{0,1\}^n} \int_{\mathbb{S}} h(g_{\mathbf{a}}(x)) d\mu(x). $$

In spirit, this identity is a consequence of the autosimilarity of $\mu$ (which itself is a consequence of the autosimilarity of $\psi$, encoded by the fact that it is a conjugacy between expanding maps). It will allows us to work on small scales with control, as the maps $g_{\mathbf{a}}$ \say{zoom in} while respecting the structure of $\psi$. \\

We define some notations for \say{orders of magnitude}. If there exists a constant $C>0$ independent of $n$ such that $a_n \leq C b_n$, then we write $a_n \lesssim b_n$. If $a_n \lesssim b_n \lesssim a_n$, we denote $a_n \simeq b_n$. If there exist $C,\alpha>0$, independent of $n$ and $\delta$, such that $ C^{-1} e^{- \alpha \delta n} a_n  \leq b_n \leq C a_n e^{\alpha \delta n} $, then we denote it by $a_n \sim b_n$. (Recall that $\delta$ measure the $C^{1+\text{Lip}}$ distance between $f$ and $f_0$). Then:

\begin{lemma} 
Recall that the perturbation $f$ is supposed to satisfy $\|f-f_0\|_{C^{1+\text{Lip}}}<\delta$. The following order of magnitude holds, for $n \geq 1$ and $\mathbf{a} \in \{0,1\}$:
$$ g_{\mathbf{a}}' \sim 2^{-n} \ , \ \text{diam}(S_\mathbf{a}) \sim 2^{-n}.$$
Moreover, $ \mu(S_\mathbf{a}) = 2^{-n} $.
\end{lemma}

The proof is straightforward, as $(f^n)' \sim 2^n$ (denoting $f^n := f \circ \dots \circ f$). The second estimate is a consequence from the first, using the mean value theorem. The last equality is by definition of $\mu$ and $S_\mathbf{a}$. \\

Finally, we prove a nonconcentration estimate for $\psi$. 

\begin{lemma}
There exists $C,\delta_{\mu}>0$ such that:

$$ \forall x \in \mathbb{S}, \ \forall r>0, \ \mu\left( [x-r,x+r] \right) \leq C r^{\delta_\mu} . $$
\end{lemma}

This can be rewritten as $$ \lambda\left( \left\{ y \in \mathbb{S}, \ \psi(y) \in [x-r,x+r] \right\} \right) \leq C r^{\delta_\mu} ,$$
where $\lambda$ denotes the Lebesgue measure on the circle, which explains why we can see this estimate on $\mu$ as a nonconcentration estimate on $\psi$.

\begin{proof}

Fix $x \in \mathbb{S}$ and $r>0$ small enough. We have, for any $n \geq 1$:

$$ \mu\left( [x-r,x+r] \right) \leq \underset{S_{\mathbf{a}} \cap [x-r,x+r] \neq \emptyset }{\sum_{\mathbf{a} \in \{0,1\}^n}} \mu(S_{\mathbf{a}}) = 2^{-n} \cdot \#\left\{\mathbf{a} \in \{0,1\}^n, \ S_{\mathbf{a}} \cap [x-r,x+r] \neq \emptyset \right\} $$

Recall that $ \text{diam}(S_\mathbf{a}) \sim 2^{-n}$. In particular, there exists $C,\alpha>0$ such that $\text{diam}(S_\mathbf{a}) \geq C 2^{-n} e^{-\delta \alpha n}$. Choosing $n(r) := \lfloor \frac{ |\ln(r)| }{\ln 2 + \delta \alpha  }\rfloor$ yields $ \text{diam}(S_\mathbf{a}) \geq C r $, so that

$$ \#\left\{\mathbf{a} \in \{0,1\}^n, \ S_{\mathbf{a}} \cap [x-r,x+r] \neq \emptyset \right\} \leq 4 C, $$
and so 
$$ \mu\left( [x-r,x+r] \right) \leq 4 C \cdot 2^{-n(r)} \leq C' r^{\delta_\mu} $$
for some $C'>0$ and for $ \delta_\mu(\delta) := (1-\delta \alpha/ \ln 2)^{-1} < 1 $. Notice that $\delta_\mu$ approaches one as $\delta$ gets smaller. \end{proof}

\section{From the continuous to the discrete}

We are ready to reduce our Van Der Corput lemma to a \say{nonlinearity estimate}. The goal of this section is to approximate the integral by a finite sum of exponential, which will be controlled by a powerful theorem of Bourgain from additive combinatorics, as soon as those nonlinearity estimates are checked. In this section, 5 quantities will be at play: $\xi$, $n$, $k$, $\varepsilon_0$ and $\delta$. The only two variables are $\xi$ and $n$, and they are related by a relation of the form $n \simeq \ln |\xi|$. The quantities $k$, $\varepsilon_0$ are constant parameters that will be fixed in section 4 while applying Theorem 4.1. The parameter $\delta$ will be chosen small before $\varepsilon_0$. \\

Our goal is to prove a bound of the form
$$ \forall \xi, \ \left| \widehat{\mu}(\xi) \right| \lesssim |\xi|^{- \rho}. $$

In the next lemma, we will consider a family of words $\mathbf{a}_{j}  \in \{0,1\}^{n}$, and we will denote their concatenation $\mathbf{A} := \mathbf{a}_0 \dots \mathbf{a}_k \in \{0,1\}^{(k+1)n}$. Same for words $\mathbf{b}_j \in \{0,1\}$ and their concatenation $\mathbf{B} := \mathbf{b}_1 \dots \mathbf{b}_k \in \{0,1\}^{kn}$. This section is devoted to the proof of the following reduction.

\begin{lemma}

Fix some $\varepsilon_0>0$ small enough. Define, for $j=1, \dots, k $, $\mathbf{A} = \mathbf{a}_0 \dots \mathbf{a}_k \in \{ 0,1\}^{(k+1)n}$ and $\mathbf{b} \in \{0,1\}^n$ : $$ \zeta_{j,\mathbf{A}}(\mathbf{b}) := 4^{kn} g_{\mathbf{a}_{j-1} \mathbf{b}}'\left( x_{\mathbf{a}_j} \right) \sim 1 $$
where $x_{\mathbf{a}} := g_{\mathbf{a}}(0) \in S_{\mathbf{a}}$. Then, for $|\xi|$ large enough, the following holds: $$|\widehat{\mu}(\xi)|^2 \lesssim e^{- \varepsilon_0  \delta_\mu n / 4} + 2^{-(k+1)n} \sum_{\mathbf{A} \in \{0,1\}^{(k+1)n}} \sup_{\eta \in [e^{\varepsilon_0 n/2}, e^{2 \varepsilon_0 n}]} 2^{-kn} \left| \sum_{\mathbf{B} \in \{0,1\}^{kn}} e^{ i \eta \zeta_{\mathbf{A},1}(\mathbf{b}_1) \dots \zeta_{\mathbf{A},k}(\mathbf{b}_k) } \right|,$$

where $n := \left\lfloor \frac{(\ln |\xi|)}{ (2k+2) \ln 2 - \varepsilon_0} \right\rfloor$.

\end{lemma}

\begin{proof}

First of all, using the autosimilarity formula for $\mu$, we get for any integer $N$:

$$ \int_0^1 e^{i \xi \psi(x)} dx = \int_0^1 e^{i \xi x} d\mu(x) $$
$$ = 2^{-N} \sum_{\mathbf{C} \in \{0,1\}^{N} } \int_{\mathbb{S}}  e^{i \xi g_{\mathbf{C}}(x)} d\mu(x) . $$

The actual value of $g_{\mathbf{C}}$ isn't important to us, the only relevant information for Fourier decay is its non-concentration. Hence, we are encouraged to use the Cauchy-Schwarz inequality like so:

$$ |\widehat{\mu}(\xi)|^2 \leq 2^{-N} \sum_{\mathbf{C} \in \{0,1\}^{N}} \left| \int_{\mathbb{S}}  e^{i \xi g_{\mathbf{C}}(x)} d\mu(x) \right|^2$$
$$ = 2^{-N} \sum_{\mathbf{C} \in \{0,1\}^{N}} \iint_{\mathbb{S}\times \mathbb{S}}  e^{i \xi \left(g_{\mathbf{C}}(x) - g_{\mathbf{C}}(y) \right)} d\mu(x) d\mu(y) .$$

Now, choose $N := (2k+1) n$, and set $\mathbf{C} = \mathbf{a}_0 \mathbf{b}_1 \mathbf{a}_1 \dots \mathbf{a}_{k-1} \mathbf{b}_k \mathbf{a}_k$, where $\mathbf{a}_i, \mathbf{b}_i \in \{0,1\}^n$. In a more compact fashion, we will denote $\mathbf{A} := \mathbf{a}_0 \mathbf{a}_1 \dots \mathbf{a}_k \in \{0,1\}^{(k+1)n} $, $ \mathbf{B} := \mathbf{b}_1 \dots \mathbf{b}_k \in \{0,1\}^{kn} $, and $\mathbf{A} * \mathbf{B} := \mathbf{C}$. This gives:

$$ |\widehat{\mu}(\xi)|^2 \lesssim 2^{-(2k+1)n} \sum_{\mathbf{A},\mathbf{B}} \iint_{\mathbb{S}\times \mathbb{S}}  e^{i \xi \left(g_{\mathbf{A*B}}(x) - g_{\mathbf{A*B}}(y) \right)} d\mu(x) d\mu(y). $$

Now, we will carefully linearize the phase. Define $\mathbf{A} \# \mathbf{B} = \mathbf{a}_0 \mathbf{b}_1 \mathbf{a}_1 \dots \mathbf{a}_{k-1} \mathbf{b}_k$. Notice that, by the mean value theorem, for all $x,y \in [0,1)$, there exists $z \in [0,1)$ such that

$$ g_{\mathbf{A*B}}(x) - g_{\mathbf{A*B}}(y) = g_{\mathbf{A\#B}}'(z)(\widehat{x}-\widehat{y}) $$

where $\widehat{x} := g_{\mathbf{a}_k}(x)$ and $\widehat{y} := g_{\mathbf{a}_k}(y)$. 
The main idea is that $g_{\mathbf{A\#B}}'(z)$ can be written as a product of $k$ functions, and this will allow us to apply the \say{sum product-phenomenon} to conclude (see section 4). We need to renormalize appropriately those functions. Define $$ \zeta_{\mathbf{A},j}(\mathbf{b}) := 4^n g_{\mathbf{a}_{j-1} \mathbf{b}}'( x_{\mathbf{a}_j} ) \sim 1 ,$$
where $x_{\mathbf{a}} := g_{\mathbf{a}}(0) \in S_{\mathbf{a}}$. The fact that $f'$ is Lipschitz gives us the following bounds:

$$ \left| 2^{-(2k+1)n} \sum_{\mathbf{A},\mathbf{B}} \iint_{\mathbb{S}\times \mathbb{S}} \left(  e^{i \xi \left(g_{\mathbf{A*B}}(x) - g_{\mathbf{A*B}}(y) - 4^{-kn} \zeta_{\mathbf{A},1}(\mathbf{b}_1) \dots \zeta_{\mathbf{A},k}(\mathbf{b}_k) (\widehat{x}-\widehat{y}) \right)} - 1 \right) d\mu(x) d\mu(y) \right| $$
$$ \lesssim |\xi|  2^{-(2k+1)n} \sum_{\mathbf{A},\mathbf{B}} \left| g_{\mathbf{A*B}}(x) - g_{\mathbf{A*B}}(y) - 4^{-kn} \zeta_{\mathbf{A},1}(\mathbf{b}_1) \dots \zeta_{\mathbf{A},k}(\mathbf{b}_k) (\widehat{x}-\widehat{y}) \right|  $$
$$ \lesssim  e^{\alpha \delta n} 2^{-(2k+2)n} |\xi| $$
for some $\alpha > 0$. This encourages us to relate $\xi$ and $n$ so that $$ |\xi| \simeq 2^{(2k+2)n} e^{- \varepsilon_0 n} $$
for some $\varepsilon_0 >0 $ small enough that will be fixed later.
This choice allows us to write, if $\delta>0$ is small enough, 

$$ |\widehat{\mu}(\xi)|^2 \lesssim  e^{- \varepsilon_0 n / 2} + 2^{-(2k+1)n} \sum_{\mathbf{A},\mathbf{B}} \iint_{\mathbb{S}\times \mathbb{S}}  e^{ i \xi 4^{kn} (\widehat{x} - \widehat{y}) \zeta_{\mathbf{A},1}(\mathbf{b}_1) \dots \zeta_{\mathbf{A},k}(\mathbf{b}_k) } d\mu(x) d\mu(y), $$

so that we may now work on the integral on the right side. Define $ \eta_{\mathbf{A}}(x,y) := \xi 4^{-k n} (\widehat{x} - \widehat{y}). $ The mean value Theorem gives us bounds of the form
$$ e^{\varepsilon_0 n} e^{-\alpha \delta n} (x-y) \lesssim |\eta_{\mathbf{A}}(x,y)| \lesssim e^{2 \varepsilon_0 n} .$$

To conclude, we just need to control the diagonal part of the integral. This is easily done using Lemma 2.2, as follows:

$$  2^{-(2k+1)n} \Big{|} \sum_{\mathbf{A},\mathbf{B}} \iint_{\left\{ |x-y| \leq e^{-(\varepsilon_0/2 - \alpha \delta) n} \right\}}  e^{ i\eta_{\mathbf{A}}(x,y) \zeta_{\mathbf{A},1}(\mathbf{b}_1) \dots \zeta_{\mathbf{A},k}(\mathbf{b}_k) } d\mu(x) d\mu(y) \Big{|} $$
$$ \lesssim \mu \otimes \mu \left( \left\{ (x,y) \in \mathbb{S} \times \mathbb{S} , \ |x-y| \leq e^{-(\varepsilon_0/2 - \alpha \delta) n} \right\} \right) $$ $$\lesssim e^{-(\varepsilon_0/2-\alpha \delta) \delta_\mu n} \lesssim e^{-\varepsilon_0 \delta_\mu n/4} .$$

So that now we may write, denoting by $D := \{ (x,y) \in \mathbb{R}^2 \ , \ |x-y| \leq e^{-(\varepsilon_0/2 - \alpha \delta)n} \}$ the previous neighborhood of the diagonal:
$$ |\widehat{\mu}(\xi)|^2 \lesssim  e^{- \varepsilon_0 n / 2} + e^{- \varepsilon_0  \delta_\mu n / 4} + 2^{-(2k+1)n} \Big{|} \sum_{\mathbf{A},\mathbf{B}} \iint_{\mathbb{S}\times \mathbb{S} \setminus D}  e^{ i \eta_\mathbf{A}(x,y) \zeta_{\mathbf{A},1}(\mathbf{b}_1) \dots \zeta_{\mathbf{A},k}(\mathbf{b}_k) } d\mu(x) d\mu(y) \Big{|} $$ 
$$\lesssim e^{- \varepsilon_0  \delta_\mu n / 4} + 2^{-(2k+1)n} \sum_{\mathbf{A} \in \{0,1\}^{(k+1)n}} \iint_{\mathbb{S}\times \mathbb{S} \setminus D} \left| \sum_{\mathbf{B} \in \{0,1\}^{kn}} e^{ i \eta_{\mathbf{A}}(x,y) \zeta_{\mathbf{A},1}(\mathbf{b}_1) \dots \zeta_{\mathbf{A},k}(\mathbf{b}_k) } \right| d\mu(x) d\mu(y) $$
$$ \lesssim e^{- \varepsilon_0  \delta_\mu n / 4} + 2^{-(k+1)n} \sum_{\mathbf{A} \in \{0,1\}^{(k+1)n}} \sup_{|\eta| \in [e^{\varepsilon_0 n/2}, e^{2 \varepsilon_0 n}]} 2^{-kn} \left| \sum_{\mathbf{B} \in \{0,1\}^{kn}} e^{ i \eta \zeta_{\mathbf{A},1}(\mathbf{b}_1) \dots \zeta_{\mathbf{A},k}(\mathbf{b}_k) } \right| .$$

\end{proof}

\section{The sum product phenomenon}

To conclusion of the proof will be a consequence of the following powerful Theorem of Bourgain:

\begin{theorem}[Sum-product phenomenon]

Fix $0 < \gamma < 1$. There exist $k \in \mathbb{N}$ and $\varepsilon_1 > 0$ depending only on $\gamma$ such that the following holds for $\eta \in \mathbb{R}$ large enough. Let $\mathcal{Z}$ be a finite set, and fix some maps $\zeta_j : \mathcal{Z} \rightarrow \mathbb{R} $, $j = 1, \dots , k$, such that, for all $j$:
$$ \forall \mathbf{b} \in \mathcal{Z}, \ |\eta|^{-\varepsilon_1/2} \leq |\zeta_j(\mathbf{b})| \leq |\eta|^{\varepsilon_1/2} $$ and 
$$\forall \sigma  \in [ | \eta |^{-2} , |\eta|^{- \varepsilon_1} ], \quad  \# \{ (\mathbf{b},\mathbf{c}) \in \mathcal{Z}^2 , \ |\zeta_j(\mathbf{b})-\zeta_j(\mathbf{c})| \leq \sigma \} \leq \left( \# \mathcal{Z} \right)^2 \sigma^{\gamma}. \quad (*)$$
Then there exists a constant $c > 0$ depending only on $\gamma$ such that 
$$ \left| \frac{1}{\# \mathcal{Z}^k} \sum_{\mathbf{b}_1, \dots \mathbf{b}_k \in \mathcal{Z}} \exp\left( i \eta \zeta_1(\mathbf{b}_1) \dots \zeta_k(\mathbf{b}_k) \right)   \right| \leq c | \eta | ^{ - \varepsilon_1 } .$$

\end{theorem}

Theorem 4.2 is an example of results called \say{sum-product phenomenons}. The main mechanism behind it is the fact that \say{multiplicative structure} seems to behave chaotically from an additive point of view. Thus, enough multiplicative structure in the phase will produce additive pseudo-randomness, which might implies some cancellations - and it does (at some scale), as soon as the phase isn't too much concentrated. See \cite{Gr09} for a gentle introduction to those ideas. \\

This version is an easy corollary of Proposition 3.2 in \cite{BD17}. A similar statement in a two dimensional setting can be found in \cite{Le21}, the proof in the one-dimensional case is analogous. The original sum-product phenomenon of Bourgain in \cite{BD17} only deals with maps $\zeta_j$ that takes images away from 0 and infinity, but such an adaptation was already used implicitely in the work of Salhsten and Steven \cite{SS20}. \\

Our goal is to apply Theorem 4.2 with $ \mathcal{Z} := \{0,1\}^n $ and $\zeta_j := \zeta_{\mathbf{A},j}$. The fact that $\zeta_{\mathbf{A},j} \sim 1$ means that there exists a constant $\alpha>0$ such that $e^{-\alpha \delta n} \leq |\zeta_{\mathbf{A},j}(\mathbf{b})| \leq e^{\alpha \delta n}$, which gives the bound $|\eta|^{-\varepsilon_1/2} \leq |\zeta_{\mathbf{A},j}(\mathbf{b})| \leq |\eta|^{\varepsilon_1/2}$ for $n$ large enough and $\delta$ small enough. The only difficult requirement to check is the \say{non concentration hypothesis} $(*)$ (which is a non-linearity estimate on $f$). In $\cite{SS20}$ and $\cite{Le21}$, this estimate is checked using \emph{Dolgopyat's estimates}, such as found in $\cite{Do98}$.
In $\cite{LNP19}$, the non-concentration estimates are checked using regularity estimates for stationary measures of random walks. In the early work of Bourgain and Dyatlov $\cite{BD17}$, the nonconcentration estimates are checked directly, without the need of any additional technology. To get an elementary conclusion, we choose to specify a particular family of perturbation on the doubling map for which some explicit computations can be done. We postpone to the next section the proof of the following

\begin{lemma}
Fix $\gamma:=1/100$. Then Theorem 4.2 fixes some $k \in \mathbb{N}$ and $\varepsilon_1 \in ]0,1[$. Fix $\varepsilon_0 := 1/20$ and $\delta \in ]0, \varepsilon_0 \varepsilon_1/2000[$. We call a block $\mathbf{A}=\mathbf{a_0} \dots \mathbf{a}_k \in \{0,1\}^{(k+1)n}$ regular if for all $j \in \llbracket 1,k\rrbracket$: 
$$\forall \sigma  \in [ e^{- 4 \varepsilon_0 n} , e^{-\varepsilon_0 \varepsilon_1 n/2} ], \quad \# \{ (\mathbf{b},\mathbf{c}) \in (\{0,1\}^{n})^2 , \  |\zeta_{\mathbf{A},j}(\mathbf{b})  - \zeta_{\mathbf{A},j}(\mathbf{c})| \leq \sigma \} \leq 4^n \sigma^{1/100}.$$

 Denote the set of regular blocks by $\mathcal{R}^{k+1}_n$. Then, \underline{for our particular perturbation $f_\delta$}, most blocks are regular:

$$ 2^{-(k+1)n} \# \left( \{0,1\}^{(k+1)n} \setminus \mathcal{R}_n^{k+1} \right) \lesssim \delta^{-1} e^{- \varepsilon_0 \varepsilon_1 n/400} .$$
\end{lemma}

This allows us to conclude the proof of Theorem 1.6: indeed, by Lemma 3.1, we already know that
$$ \left| \int_0^1 e^{i \xi \psi(x)} dx \right|^2 \lesssim e^{- \varepsilon_0 \delta_\mu n/4} + 2^{-(k+1)n} \sum_{\mathbf{A} \in \{0,1\}^{(k+1)n}} \sup_{\eta \in [e^{\varepsilon_0 n/2}, e^{2 \varepsilon_0 n}]} 2^{-kn} \left| \sum_{\mathbf{B} \in \{0,1\}^{kn}} e^{ i \eta \zeta_{\mathbf{A},1}(\mathbf{b}_1) \dots \zeta_{\mathbf{A},k}(\mathbf{b}_k) } \right| .$$

Using the previous bound yields:

$$ \left| \int_0^1 e^{i \xi \psi(x)} dx \right|^2 \leq C_{\delta}\left( e^{- \varepsilon_0 \delta_\mu n/4} +  e^{-\varepsilon_0 \varepsilon_1 n/400} \right) $$ $$ + 2^{-{(k+1)n}} \sum_{\mathbf{A} \in \mathcal{R}_n^{k+1}} \sup_{\eta \in [e^{\varepsilon_0 n/2}, e^{2 \varepsilon_0 n}]} 2^{-kn} \left| \sum_{\mathbf{B} \in \{0,1\}^{kn}} e^{ i \eta \zeta_{\mathbf{A},1}(\mathbf{b}_1) \dots \zeta_{\mathbf{A},k}(\mathbf{b}_k) } \right|.$$

We then use that all regular blocks $\mathbf{A}$ produces maps $\zeta_{\mathbf{A},j}$ that all satisfies the non concentration hypothesis required to apply Theorem 4.2. This gives the exponential bound:

$$  \left| \int_0^1 e^{i \xi \psi(x)} dx \right|^2 \leq C_{\delta}\left( e^{- \varepsilon_0 \delta_\mu n/4} +  e^{-\varepsilon_0 \varepsilon_1 n/400} + e^{-\varepsilon_0 \varepsilon_1 n/2} \right). $$

Notice the following interesting fact: since $\delta_\mu$ approach one as the perturbation gets smaller, we see (in our particular case of a carefully chosen perturbation of the doubling map) that the exponent of decay might be chosen \underline{constant in $\delta$} if $\delta$ is small enough.

Recalling that $n := \left\lfloor \frac{(\ln |\xi|)}{ (2k+2) \ln 2 - \varepsilon_0} \right\rfloor$ then gives 
$$ \left| \int_0^1 e^{i \xi \psi(x)} dx \right| \leq C_\delta |\xi|^{-\rho} $$

For some $\rho>0$, constant in $\delta$. The fact that we get a constant $C_\delta \geq C \delta^{-1}$ in front of our power decay is an artefact of our method: the sum product estimates needs some nonlinearity to holds, and $f_\delta$ is \say{more linear} as $\delta$ approaches zero. This answers, in our very particular case, a question found in $\cite{SS20}$ about the dependence of the exponent $\rho$ on the dynamics.

\section{The non-concentration estimates}

In this section we recall the explicit family of perturbations of the doubling map that allows us to check the non-concentration hypothesis. For any periodic function $\Phi : \mathbb{S} \rightarrow \mathbb{R}$, we are going to construct a perturbation $ f $ of the doubling map so that $\ln f' = c_0 + c_1 \cdot \Phi$.

\begin{definition}

Let $\Phi:\mathbb{S} \rightarrow \mathbb{R}$ be $1$-periodic, $1$-Lipschitz, and with absolute value bounded by $1$. Then, for $\delta\in (0,1)$, set
$$ \varphi_\delta(x) := z_\delta \int_0^x e^{\delta \Phi(t)} dt. $$
where $$ z_\delta^{-1} := \int_{0}^{1} e^{\delta \Phi(t)} dt .$$

We see that $\varphi$ is a perturbation of the identity that factors into a $C^{1+\text{Lip}}$-diffeomorphism of the circle. More precisely, there exists a constant $C>0$ such that $ \| \varphi_\delta - I_d \|_{C^{1+\text{Lip}}} \leq C \delta. $ Moreover, for all $x$, $ x e^{- 2 \delta} \leq \varphi_\delta(x) \leq x e^{2\delta} $, and $ e^{-2 \delta} \leq \varphi_\delta'(x) \leq e^{2 \delta}$.

\end{definition}

\begin{definition}

Our perturbation of the doubling map is defined as follows: for some \underline{fixed} $\delta > 0$, set $$ f_\delta(x) := \varphi_\delta(2x) .$$
The parameter $\delta$ will be taken small enough at the end of the section. Notice that $\varphi_0(x)=x$ and so $f_0$ is the doubling map. We will omit $\delta$ and write $f,\varphi$ instead of $f_\delta, \phi_\delta$ for the rest of the section. We define the inverse branches for $a \in \{0,1\}$ by $ g_a(x) := g_a^{(0)}(\varphi^{-1}(x)) $. We define, for a word $\mathbf{a} \in \{0,1\}^n$, $g_{\mathbf{a}} := g_{a_1 \dots a_n}$. Finally, set $S_\mathbf{a} := g_{\mathbf{a}}(\mathbb{S})$.

\end{definition}

\begin{remark}

Notice that $f$ was constructed such that $S_0 = [0,1/2)$ and $S_1 =[1/2,1)$. 
\end{remark}

\begin{lemma}
Let $\mathbf{a} \in \{0,1\}^n$.
Then $$ - \ln g_{\mathbf{a}}' =  n (\ln 2 + \ln z_\delta) + \delta \cdot S_n \Phi \circ g_{\mathbf{a}} ,$$

where $S_n \Phi := \sum_{k=0}^{n-1} \Phi \circ f^k$. In particular, $2^n g_{\mathbf{a}}' \in [ e^{- 2 \delta n}, e^{2 \delta n} ]$, and $\left|\frac{g_{\mathbf{a}}'(x)}{g_{\mathbf{a}}'(y)} - 1\right| \leq 3 \delta e^{2 \delta n} |x-y| $.

\end{lemma}

\begin{proof}

We see that
$$ \ln g_{\mathbf{a}}' = - \ln (f^n)'\circ g_{\mathbf{a}} =  -\sum_{k=0}^{n-1} (\ln f') \circ f^k \circ g_{\mathbf{a}} .$$
Moreover, $\ln f'(x) = \ln \varphi'(2x) + \ln 2$, and $\ln \varphi' = \ln z_\delta + \delta \cdot \Phi$. The first estimate is easy since $|\Phi|_\infty \leq 1$ and $|\ln z_\delta| \leq \delta$. The second estimate can be checked as follows:

$$ \left|\frac{g_{\mathbf{a}}'(x)}{g_{\mathbf{a}}'(y)} - 1\right|  = \left| e^{\ln g_{\mathbf{a}}'(x) - \ln g_\mathbf{a}'(y) } - 1\right|  \leq e^{ 2 n \delta} \left| \ln g_{\mathbf{a}}'(x) - \ln g_\mathbf{a}'(y)\right| $$ $$ \leq \delta e^{2 n \delta} \sum_{j=1}^n \left| \Phi(g_{a_j \dots a_n} x) -\Phi(g_{a_j \dots a_n} y) \right|  \leq \delta e^{ 2 \delta n} \sum_{j=1}^n (2/3)^{n} |x-y| \leq 2 \delta e^{2 \delta n} |x-y| .$$ \end{proof}

\begin{definition}
We let $\Phi:\mathbb{R} \rightarrow \mathbb{R}$ be the $1$-periodic, Lipschitz and bounded by one function defined by
$$ \forall x \in [0,1/2], \ \Phi(x) := x - \frac{1}{4} \ , \ \text{and} \ \forall x \in [1/2,1], \ \Phi(x) := \frac{3}{4}-x .$$

It is differentiable on $\mathbb{R} \setminus \frac{1}{2} \mathbb{Z} $, with derivative $1$ on $\mathbb{S}_0$ and $-1$ on $\mathbb{S}_1$. In particular, $\Phi' \circ g_a$ is naturally extended as a constant map on $[0,1]$.

\end{definition}

Our goal is to prove Lemma 4.2 for this choice of $\Phi$. The idea of the proof can be stated in two main steps.
\begin{enumerate}
    \item The nonconcentration hypothesis can be rewritten in terms of a non-concentration estimate involving Birkhoff sums involving $\Phi$, namely $S_n \Phi \circ g_{\mathbf{a}}$.
    \item To check that those Birkhoff sums doesn't concentrate too much, we show that the derivatives $\left(S_n \Phi \circ g_{\mathbf{a}} - S_n \Phi \circ g_{\mathbf{b}} \right)'$ are often away from zero.
\end{enumerate}
We begin by step 2. To this end, the following preliminary lemma is helpful. 

\begin{lemma}
 Suppose that $\delta<1/200$. Let $E$ be a non-empty set. Let $n \geq 1$, and denote by $P(n)$ the following property: 
\begin{itemize}
    \item Let $\sigma>0$ be a scale factor. For all $x \in E$, for all $i \in \llbracket 1,n \rrbracket$, and for all $\widehat{\Phi} \in \{-1,1\}^i$, choose $\rho_{i}(\widehat{\Phi},x) \in \ ] \frac{1}{2} e^{-2 \delta},\frac{1}{2} e^{2 \delta} [ $ with small fluctuations:
    $$ \forall x,y \in E, \ \left|\rho_i(\widehat{\Phi},x) - \rho_i(\widehat{\Phi},y) \right| \leq \left( \frac{e^{2 \delta}}{2} \right)^{i} \sigma^3 .$$ For any $j \geq i$ and any word $\widehat{\Phi} \in \{0,1\}^j$, set $\kappa^i_{\widehat{\Phi}}(x) := \rho_1(\widehat{\Phi}_1,x) \rho_2(\widehat{\Phi}_1 \widehat{\Phi}_2,x) \dots \rho_i(\widehat{\Phi}_1 \dots \widehat{\Phi}_i,x)$. Define the map $X_n^x : \{-1,1\}^n \longrightarrow \mathbb{R}$ by
$$ X_n^x(\widehat{\Phi}) := \sum_{i=1}^n {\widehat{\Phi}_i} \cdot {\kappa^i_{\widehat{\Phi}}}(x). $$
Finally, choose a target $a : E \rightarrow \mathbb{R}$ with small fluctuations: $\forall x,y \in E, |a(x)-a(y)| \leq \sigma^3$. Then, the following uniform anti-concentration estimate holds:
$$ 2^{-n} \# \left\{ \widehat{\Phi} \in \{-1,1\}^{n} \ | \ \exists x \in E, \ X_n^x(\widehat{\Phi}) \in [a(x)-\sigma,a(x)+\sigma] \right\} \leq \left(\frac{4}{3}\right)^2 \sigma^{\alpha_0} e^{2 \alpha_0 \delta n} + 2 \left( \frac{3}{4} \right)^{n/2}. $$  
where $\alpha_0 := 1- \frac{\ln 3}{\ln 4} > 1/5$.
\end{itemize}
Then $P(n)$ is true.
\end{lemma}

\begin{proof}
The idea is to use the \say{uniform in $x$} fractal geometry of the sets $T_n^x := \{ X_n^x(\widehat{\Phi}) , \widehat{\Phi} \in \{-1,1\}^n \}$ to relate the behavior of $T_n^x$ to the behavior of $T_{j}^x$, $j<n$. We will thus prove that $P(n)$ is true for all $n \in \mathbb{N}^*$ by induction on $n$. If $n=1,2$, the estimate holds since $1 \leq 2 \cdot (3/4)$. Now, let $n \geq 3$ and suppose that the estimate holds for all $j<n$. We notice that
$$ X_n^x(\widehat{\Phi}) = \widehat{\Phi}_1 \kappa_{\widehat{\Phi}}^1(x) + \widehat{\Phi}_2 \kappa_{\widehat{\Phi}}^2(x) + \kappa_{\widehat{\Phi}}^2(x) \tilde{X}_{\widehat{\Phi}_1 \widehat{\Phi}_2,n-2}^x(\tilde{\Phi}) $$
where $\tilde{\Phi}=\widehat{\Phi}_3 \dots \widehat{\Phi}_n \in \{-1,1\}^{n-2}$, and $ \tilde{X}_{\widehat{\Phi}_1 \widehat{\Phi}_2,n-2}^x(\widehat{\Psi}) := \sum_{i=1}^{n-2} \widehat{\Psi}_i \tilde{\kappa}_{\widehat{\Phi}_1 \widehat{\Phi}_2,\widehat{\Psi}}^i(x) $
with $\tilde{\rho}_{\widehat{\Phi}_1 \widehat{\Phi}_2,i}(\widehat{\Psi},x)=\rho_{i+2}(\widehat{\Phi}_1 \widehat{\Phi}_2 \widehat{\Psi},x)$ and $\tilde{\kappa}^i_{\widehat{\Phi}_1 \widehat{\Phi}_2,\widehat{\Psi}}(x) := \tilde{\rho}_{\widehat{\Phi}_1 \widehat{\Phi}_2,1}(\widehat{\Psi},x) \dots \tilde{\rho}_{\widehat{\Phi}_1 \widehat{\Phi}_2,i}(\widehat{\Psi},x)$. The idea is to apply the case $n-2$ with $\tilde{X}_{\widehat{\Phi}_1 \widehat{\Phi}_2,n-2}^x$.
Let $\sigma>0$. Notice that, since $(4/3)^2 (1/8)^{\alpha_0} = 1$, the estimate is true for $\sigma \geq 1/8$. So suppose $\sigma<1/8$. Moreover, notice that since $|X_n^x(\widehat{\Phi})| \leq 2$, we can suppose $\|a\|_\infty \leq 3$: indeed, if this is not the case, then our cardinal is zero. Now that we have reduced the lemma in the interesting cases, write:

$$ 2^{-n} \# \left\{ \widehat{\Phi} \in \{-1,1\}^{n} \ | \ \exists x \in E, \ X_n^x(\widehat{\Phi}) \in [a(x)-\sigma,a(x)+\sigma] \right\} $$ $$ = 2^{-n} \# \left\{ \widehat{\Phi} \in \{-1,1\}^{n} \ | \ \exists x \in E, \  \widehat{\Phi}_1 \kappa_{\widehat{\Phi}}^1(x) + \widehat{\Phi}_2 \kappa_{\widehat{\Phi}}^2(x) + \kappa_{\widehat{\Phi}}^2(x) \tilde{X}_{\widehat{\Phi}_1 \widehat{\Phi}_2,n-2}^x(\tilde{\Phi})  \in [a(x)-\sigma,a(x)+\sigma] \right\} $$
$$ = \sum_{(i_1,i_2) \in \{\pm 1\}} 2^{-n} \# \left\{ \tilde{\Phi} \in \{\pm 1\}^{n-2} | \ \exists x \in E, \  i_1 \kappa_{i_1}^1(x) +  \kappa_{i_1 i_2}^2(x)( i_2 + \tilde{X}_{i_1 i_2,n-2}^x(\tilde{\Phi}))  \in [a(x)-\sigma,a(x)+\sigma] \right\}. $$
We then make the following claim: of all the four combinations possibles for $(i_1,i_2)$, only three of them allow the existence of some $x \in E$ for which $i_1 \kappa_{i_1}^1(x) + i_2 \kappa_{i_1 i_2}^2(x) + \kappa_{i_1 i_2}^2(x) \tilde{X}_{i_1 i_2,n-2}^x(\tilde{\Phi}) \in [a(x)-\sigma,a(x)+\sigma]$.
Indeed, suppose for example that there exists $x_0 \in E$ such that $a(x_0)=0$ and $i_1=i_2 = -1$. In this case, $a(x)-\sigma>-1/4$ for all $x \in E$ (recall that $\sigma<1/8$). We notice that
$$ \forall x \in E, \ i_1 \kappa_{i_1}^1(x) + i_2 \kappa_{i_1 i_2}^2(x) + \kappa_{i_1 i_2}^2(x) \tilde{X}_{i_1 i_2,n-2}^x(\tilde{\Phi}) \leq - \frac{e^{-2\delta}}{2} - \frac{e^{-4\delta}}{4} + \frac{e^{4\delta}}{4} \sum_{i=1}^{n-2} \left(\frac{e^{2\delta}}{2}\right)^i  $$
$$ \leq - \frac{e^{-2\delta}}{2} - \frac{e^{-4\delta}}{4} + \frac{e^{6\delta}}{8} \frac{1}{1-\frac{e^{2\delta}}{2}} < -1/4 $$
since $\delta<1/200$. It follows that $i_1 \kappa_{i_1}^1(x) + i_2 \kappa_{i_1 i_2}^2(x) + \kappa_{i_1 i_2}^2(x) \tilde{X}_{i_1 i_2,n-2}^x(\tilde{\Phi})$ can never belong to $[a(x)-\sigma,a(x)+\sigma]$. \emph{A fortiori}, it can never meet $[a(x)-\sigma,a(x)+\sigma]$ as soon as there exists some $x_0 \in E$ such that $a(x_0) \geq 0$. Symmetrically, $[a(x)-\sigma,a(x)+\sigma]$ will never meet the terms with $i_1=i_2=1$ if there exists $x_0$ for which $a(x_0) \leq 0$. \\

We can then conclude the computation by justifying that $P(n-2)$ applies. Our scale factor will be $\tilde{\sigma} := 4 e^{4 \delta} \sigma$, and the target $\tilde{a}_{i_1 i_2}(x) := \left(\kappa^{2}_{i_1 i_2}(x)\right)^{-1} \left( a(x) -i_1 \kappa_{i_1}^1(x) - i_2 \kappa_{i_1 i_2}^2(x) \right)$. Notice that
$$ \forall x,y \in E, \left| \left( a(x) -i_1 \kappa_{i_1}^1(x) - i_2 \kappa_{i_1 i_2}^2(x) \right)-\left( a(y) -i_1 \kappa_{i_1}^1(y) - i_2 \kappa_{i_1 i_2}^2(y) \right) \right| \leq 3 \sigma^3  $$
and $|a(x) -i_1 \kappa_{i_1}^1(x) - i_2 \kappa_{i_1 i_2}^2(x)| \leq 5$, which gives the bound $|\tilde{a}_{i_1 i_2}(x) - \tilde{a}_{i_1 i_2}(y)| \leq 30 \sigma^3 = 30 (\frac{\tilde{\sigma}}{4 e^{4 \delta}})^3 \leq \tilde{\sigma}^3$. We also have to check that the $\tilde{\rho}_{i_1 i_1, i}$ have small fluctuations:
$$ \forall \widehat{\Psi}, \ \forall x,y \in E, \ \left| \tilde{\rho}_i(\widehat{\Psi},x) - \tilde{\rho}_i(\widehat{\Psi},y) \right|  = \left| {\rho}_{i+2}(i_1 i_2 \widehat{\Psi},x) - {\rho}_{i+2}(i_1 i_2 \widehat{\Psi},y) \right| $$ $$ \leq \left( \frac{e^{2 \delta}}{2} \right)^{i+2} \sigma^3 \leq \left( \frac{e^{2 \delta}}{2} \right)^{i} \tilde{\sigma}^3 .$$
We can then safely apply $P(n-2)$.
The desired estimate follows by using the previous claim and the induction hypothesis:

$$ 2^{-n} \# \left\{ \widehat{\Phi} \in \{-1,1\}^{n} \ | \exists x \in E, \ X_n^x(\widehat{\Phi}) \in [a(x)-\sigma,a(x)+\sigma] \right\} $$
$$  = 2^{-n} \sum_{(i_1,i_2) \in \{\pm 1\}^2} \# \Big{\{} \tilde{\Phi} \in \{\pm 1\}^{n-2}| \exists x \in E, \   \tilde{X}_{i_1 i_2,n-2}^x(\tilde{\Phi}) - \tilde{a}_{i_1 i_2}(x) \in [- \left(\kappa^{2}_{i_1 i_2}(x)\right)^{-1} \sigma, \left(\kappa^{2}_{i_1 i_2}(x)\right)^{-1} \sigma] \Big{\}} $$ 
$$ \leq \sum_{(i_1,i_2) \in \{-1,1\}} 2^{-n} \# \left\{ \tilde{\Phi} \in \{-1,1\}^{n-2} | \exists x \in E, \  \tilde{X}_{i_1 i_2, n-2}^x(\tilde{\Phi}) \in [\tilde{a}_{i_1 i_2}(x)- 4 e^{4 \delta} \sigma,\tilde{a}_{i_1 i_2}(x)+ 4 e^{4 \delta} \sigma] \right\} $$ $$ \leq 3 \cdot 2^{-2} \left( \left(\frac{4}{3}\right)^2 (4 e^{4 \delta} \sigma)^{\alpha_0} e^{2 \alpha_0 \delta (n-2)} + 2 \left(3/4\right)^{(n-2)/2} \right)  = \left(\frac{4}{3}\right)^2 \sigma^{\alpha_0} e^{2 \alpha_0 \delta n} + 2 \left(\frac{3}{4}\right)^{n/2} .$$ \end{proof}

\begin{lemma}
Let $n$ be large enough. Let $\sigma \in [e^{-5 \varepsilon_0 n},\delta^{-1} e^{-\varepsilon_0 \varepsilon_1 n/3} ]$. Define $ \tilde{n} := \left\lfloor (\log_2 \sigma)/2 \right\rfloor $. This is a slowly increasing zoom factor, scaled so that $2^{-\tilde{n}} e^{2 \delta n} \simeq \sqrt{\sigma} e^{2 \delta n} \leq ({\sigma}^{1/10})^3$. Fix any word $\mathbf{a} \in \{0,1\}^n$. The following bound holds:
$$ 2^{-2n-\tilde{n}}\#\{ (\mathbf{b},\mathbf{c}) \in (\{ 0,1 \}^{n})^2, \mathbf{d} \in \{0,1\}^{\tilde{n}}| \exists x \in S_{\mathbf{d}}, \left|\left(S_{2n} \Phi \circ g_{\mathbf{a} \mathbf{b}} - S_{2n} \Phi \circ g_{\mathbf{a} \mathbf{c}}\right)'(x)\right| \leq \sigma^{1/10} \} \leq \delta^{-1/2} \sigma^{1/50} . $$

\end{lemma}

\begin{proof}

We are going to reduce our bound to the previous lemma. Notice first that, for fixed words $\mathbf{a},\mathbf{b}$ and $\mathbf{c}$, we can compute the derivative of $S_{2n} \Phi \circ g_{\mathbf{a} \mathbf{b}} - S_{2n} \Phi \circ g_{\mathbf{a} \mathbf{c}} $ as follow:
$$ \left(S_{2n} \Phi \circ g_{\mathbf{a} \mathbf{b}} - S_{2n} \Phi \circ g_{\mathbf{a} \mathbf{c}}\right)' = g_{\mathbf{b}}' \left(S_n \circ g_{\mathbf{a}}\right)' \circ g_\mathbf{b} + \left( S_n \Phi \circ g_{\mathbf{b}} \right)' - g_{\mathbf{c}}' \left(S_n \circ g_{\mathbf{a}}\right)' \circ g_\mathbf{c} - \left( S_n \Phi \circ g_{\mathbf{c}} \right)'. $$

We see that the terms involving $\mathbf{a}$ becomes negligible. Indeed, $ \left| \left(S_n \Phi \circ g_{\mathbf{a}} \right)' \right|  \leq2 $, and $|g_{\mathbf{a}}'| \leq 2^{-n} e^{2 \delta n}$, so that
$$ \left| g_{\mathbf{b}}' \left(S_n \circ g_{\mathbf{a}}\right)' \circ g_\mathbf{b} - g_{\mathbf{b}}' \left(S_n \circ g_{\mathbf{a}}\right)' \circ g_\mathbf{c} \right| \leq 4 \cdot 2^{-n} e^{2 \delta n} \leq \sigma^{1/10} $$
for $n$ large enough, since $\delta<1/10$ and $\varepsilon_0<1$.
Hence, if $ |\left(S_{2n} \Phi \circ g_{\mathbf{a} \mathbf{b}} - S_{2n} \Phi \circ g_{\mathbf{a} \mathbf{c}}\right)'|  \leq \sigma^{1/10} $, then $\left| \left( S_n \Phi \circ g_{\mathbf{b}} - S_n \Phi \circ g_{\mathbf{c}} \right)' \right| \leq 2\sigma^{1/10} $, and it follows that $$ 2^{-2n-\tilde{n}}\#\{ (\mathbf{b},\mathbf{c}) \in (\{ 0,1 \}^{n})^2, \mathbf{d} \in \{0,1\}^{\tilde{n}} \ |  \ \exists x \in S_{\mathbf{d}}, \ \left|\left(S_{2n} \Phi \circ g_{\mathbf{a} \mathbf{b}} - S_{2n} \Phi \circ g_{\mathbf{a} \mathbf{c}}\right)'(x)\right| \leq \sigma^{1/10} \}  $$
$$ \leq  2^{-2n-\tilde{n}}\#\{ (\mathbf{b},\mathbf{c}) \in (\{ 0,1 \}^{n})^2, \mathbf{d} \in \{0,1\}^{\tilde{n}} \ |  \ \exists x \in S_{\mathbf{d}}, \ \left|\left(S_{n} \Phi \circ g_{\mathbf{b}} - S_{n} \Phi \circ g_{\mathbf{c}}\right)'(x)\right| \leq 2\sigma^{1/10} \}. $$

The derivative can be further simplified, using the special $\Phi$ that we chose.  Indeed, we see that, for any $x \in [0,1]$:
$$ \left( S_n \Phi \circ g_{\mathbf{b}} \right)'(x) = \sum_{j=1}^n \Phi'\left( g_{b_j \dots b_n}(x) \right) g_{b_j \dots b_n}'(x) = \sum_{j=1}^n \widehat{\Phi}(b_{n-j+1}) \kappa_{\mathbf{b}}^j(x) , $$
where $\widehat{\Phi}(0) := 1$ and $\widehat{\Phi}(1)=-1$, and $\kappa_{\mathbf{b}}^{n-j+1}(x) := g_{b_j \dots b_n}'(x)$ (recall Remark 5.1). 
Define $$ X_n^{x}(\mathbf{b}) := \sum_{j=1}^n \widehat{\Phi}(b_{n-j-1}) \kappa_{\mathbf{b}}^{j}(x).$$ The associated maps $\rho_i$ are $\rho_{n-j+1}(\mathbf{b},x) := g_{b_j}'(g_{b_{j+1} \dots b_n} x). $ We can then rewrite our cardinal in a more compact form, as follows:
$$ 2^{-2n-\tilde{n}}\#\left\{ (\mathbf{b},\mathbf{c}) \in (\{ 0,1 \}^n)^{2}, \mathbf{d} \in \{0,1\}^{\tilde{n}} \ | \ \exists x \in S_{\mathbf{d}}, \ \left| X_n^{x}(\mathbf{b})-X_n^x(\mathbf{c})\right| \leq 2 \sigma^{1/10} \right\}$$
$$ = 2^{-{\tilde{n}}-n} \underset{\mathbf{c} \in \{0,1\}^n}{\sum_{\mathbf{d} \in \{0,1\}^{\tilde{n}}}} 2^{-n} \#\left\{ \mathbf{b} \in \{ 0,1 \}^n \ | \ \exists x \in S_{\mathbf{d}}, \ X_n^{x}(\mathbf{b}) \in [a_\mathbf{c}(x) - 2 \sigma^{1/10}, a_{\mathbf{c}}(x) + 2 \sigma^{1/10}] \right\} $$
with $ a_\mathbf{c}(x) := X_n^x(\mathbf{c}) $. We then wish to apply the previous lemma. To this end, we check first that the $\rho_i$ have small fluctuations: since $\text{diam}(S_{\mathbf{c}}) \lesssim 2^{-\tilde{n}} e^{\delta \tilde{n}} \leq (\sigma^{1/10})^3$, and since $g_{b_{n-i+1} \dots b_n}$ is Lipschitz with constant $(e^{2 \delta}/2)^i$, we see that $\rho_i$ has small enough fluctuations (see the bounds in Lemma 5.1). We then also need to check that the target $a_\mathbf{c}(x)$ has small fluctuations. This is done using the bounds found in lemma 5.1 again:
$$ |a_{\mathbf{c}}(x) - a_{\mathbf{c}}(y)| \leq \sum_i \kappa_{\mathbf{c}}^i(y) |\kappa_{\mathbf{c}}^i(x) / \kappa_{\mathbf{c}}^i(y) -1| \lesssim \sigma^{1/2} |x-y| \leq (\sigma^{1/10})^3 |x-y| .$$
Hence Lemma 5.2 applies, and gives:

$$ 2^{-{\tilde{n}}-n} \underset{\mathbf{c} \in \{0,1\}^n}{\sum_{\mathbf{d} \in \{0,1\}^{\tilde{n}}}} 2^{-n} \#\left\{ \mathbf{b} \in \{ 0,1 \}^n \ | \ \exists x \in S_{\mathbf{d}}, \ X_n^{x}(\mathbf{b}) \in [a_\mathbf{c}(x) - 2 \sigma^{1/10}, a_{\mathbf{c}}(x) + 2 \sigma^{1/10}] \right\}$$

$$ \leq \left( \frac{4}{3} \right)^2 \left(2 \sigma^{1/10} \right)^{\alpha_0} e^{ 2 \alpha_0 \delta n } + 2 \cdot \left( \frac{3}{4} \right)^n \leq  \delta^{-1/2} \sigma^{1/50}  $$
provided that $n$ is taken large enough and using the fact that $(4/3)^2 \cdot 2 \sigma^{\alpha_0/10} e^{2 \alpha_0 \delta n} < \delta^{-1/2} 2^{-1} \sigma^{1/50}$ since $\delta<\varepsilon_0 \varepsilon_1/400$. \end{proof}

\begin{lemma}
Let $n$ be large enough. Let $\sigma \in [e^{-5 \varepsilon_0 n},\delta^{-1} e^{-\varepsilon_0 \varepsilon_1 n/3} ]$. Let $\mathbf{a} \in \{0,1\}^n$.  Then:
$$ 8^{-n} \# \left\{ (\mathbf{b}, \mathbf{c},\mathbf{d}) \in \left(\{0,1\}^n\right)^3 , \  |S_{2n} \Phi \circ g_{\mathbf{a} \mathbf{b}}(x_{\mathbf{d}}) - S_{2n} \Phi \circ g_{\mathbf{a} \mathbf{c}}(x_{\mathbf{d}}) | \leq \sigma \right\} \leq  2 \delta^{-1/2} \sigma^{1/50} $$
\end{lemma}

\begin{proof}
 We cut the word $\mathbf{d}$ in two : $\mathbf{d} := \tilde{\mathbf{d}} \widehat{\mathbf{d}}$, with $\tilde{\mathbf{d}} \in \{0,1\}^{\tilde{n}}$ and $\widehat{\mathbf{d}} \in \{0,1\}^{n-\tilde{n}}$, where $\tilde{n} := \lfloor (\log_2 \sigma)/2 \rfloor$. The desired cardinal becomes
$$ 8^{-n} \#\left\{ (\mathbf{b},\mathbf{c},\tilde{\mathbf{d}},\widehat{\mathbf{d}}) \in (\{0,1\}^n)^2 \times \{0,1\}^{\tilde{n}} \times \{0,1\}^{n-\tilde{n}}, \ |S_{2n} \Phi \circ g_{\mathbf{a} \mathbf{b}}(x_{\tilde{\mathbf{d}} \widehat{\mathbf{d}}}) - S_{2n} \Phi \circ g_{\mathbf{a} \mathbf{c}}(x_{\tilde{\mathbf{d}} \widehat{\mathbf{d}}})|\leq \sigma \right\} .$$

From there, the strategy is taken from \cite{BD17}: we argue that for most of the words $\mathbf{b},\mathbf{c},\tilde{\mathbf{d}}$, the derivative of the inner function is large enough, thus spreading the $x_{\tilde{\mathbf{d}}\widehat{\mathbf{d}}}$.
Indeed, if we denote by $D_n(\sigma^{1/10})$ the set of all $(\mathbf{b},\mathbf{c},\tilde{\mathbf{d}})$ for which there exists $x \in \mathbb{S}_{\tilde{\mathbf{d}}}$ such that $$ |\left(S_{2n} \Phi \circ g_{\mathbf{a} \mathbf{b}} - S_{2n} \Phi \circ g_{\mathbf{a} \mathbf{c}}\right)'(x)| \leq \sigma^{1/10} ,$$ then the previous lemma bounds $2^{-2n-\tilde{n}} \# D_n(\sigma^{1/10})$, and we can write:

$$  8^{-n} \#\left\{ (\mathbf{b},\mathbf{c},\tilde{\mathbf{d}},\widehat{\mathbf{d}}), \ |S_{2n} \Phi \circ g_{\mathbf{a} \mathbf{b}}(x_{\tilde{\mathbf{d}} \widehat{\mathbf{d}}}) - S_{2n} \Phi \circ g_{\mathbf{a} \mathbf{c}}(x_{\tilde{\mathbf{d}} \widehat{\mathbf{d}}})|\leq \sigma \right\} . $$
$$ \leq 8^{-n} \# \left\{ (\mathbf{b}, \mathbf{c}, \tilde{\mathbf{d}}, \widehat{\mathbf{d}}) \ | \ (\mathbf{b},\mathbf{c},\tilde{\mathbf{d}}) \notin D_n(\sigma^{1/10}), \  |S_{2n} \Phi \circ g_{\mathbf{a} \mathbf{b}}(x_{\tilde{\mathbf{d}} \widehat{\mathbf{d}}}) - S_{2n} \Phi \circ g_{\mathbf{a} \mathbf{c}}(x_{\tilde{\mathbf{d}} \widehat{\mathbf{d}}})|\leq \sigma \right\} + \delta^{-1/2} \sigma^{1/50} .$$

Now, $(\mathbf{b},\mathbf{c},\tilde{\mathbf{d}}) \notin D_n(\sigma^{1/10})$ means that $$ \inf_{S_{\tilde{\mathbf{d}}}} \left| \left( S_{2n} \Phi \circ g_{\mathbf{a} \mathbf{b}} - S_{2n} \Phi \circ g_{\mathbf{a} \mathbf{c}} \right)' \right| \geq \sigma^{1/10} .$$

It is elementary to check that for any absolutely continuous map $f : I \rightarrow \mathbb{R}$ satisfying $\inf_I f' > 0$, we have, for any interval $J$, $\text{diam}(f^{-1}(J)) \leq (\inf_I f')^{-1} \text{diam}(J)$. Hence, if $(\mathbf{b},\mathbf{c},\tilde{\mathbf{d}}) \notin D_n(\sigma^{1/10})$, denoting by $I_{\mathbf{a},\mathbf{b},\mathbf{c},\tilde{\mathbf{d}}}(\sigma) := \{ x \in \mathbb{S}_{\tilde{\mathbf{d}}}, (S_{2n} \Phi \circ g_{\mathbf{a} \mathbf{b}} - S_{2n} \Phi \circ g_{\mathbf{a} \mathbf{c}})(x) \in [-\sigma,\sigma] \}$, we have $\text{diam}(I_{\mathbf{a},\mathbf{b},\mathbf{c},\tilde{\mathbf{d}}}(\sigma)) \leq 2 \sigma^{9/10}$, and :
$$ 2^{-(n-\tilde{n})} \# \{ \widehat{\mathbf{d}} \in \{0,1\}^{n-\tilde{n}} , \ | S_{2n} \Phi \circ g_{\mathbf{a} \mathbf{b}}(x_{\tilde{\mathbf{d}}\widehat{\mathbf{d}}}) - S_{2n} \Phi \circ g_{\mathbf{a} \mathbf{c}}(x_{\tilde{\mathbf{d}}\widehat{\mathbf{d}}}) | \leq \sigma  \}   $$
$$ = 2^{-(n-\tilde{n})} \# \{ \widehat{\mathbf{d}} \in \{0,1\}^{n-\tilde{n}} ,  \ x_{\tilde{\mathbf{d}} \widehat{\mathbf{d}}}  \in I_{\mathbf{a},\mathbf{b},\mathbf{c},\tilde{\mathbf{d}}}(\sigma)  \} $$
$$ \leq 2^{-(n-\tilde{n})}\left(1+\frac{\text{diam}(I_{\mathbf{a},\mathbf{b},\mathbf{c},\tilde{\mathbf{d}}}(\sigma))}{2^{-n} e^{-4 \delta n} }\right) \leq \sigma^{1/10} $$
since the $x_{\mathbf{c}}$ are spaced out by at least $2^{-n} e^{-4 \delta n}$ from each other (Lemma 2.1), and $2^{\tilde{n}} \simeq \sigma^{-1/2}$.
\end{proof}

\begin{lemma}

Let $n$ be large enough. Let $\sigma \in [e^{-4 \varepsilon_0 n}, e^{-\varepsilon_0 \varepsilon_1 n/2}]$. Let $\mathbf{a} \in \{0,1\}^n$.  Then:
$$ 8^{-n} \# \left\{ (\mathbf{b}, \mathbf{c}, \mathbf{d}) \in \left(\{0,1\}^n\right)^3 , \ | 4^n g_{\mathbf{a} \mathbf{b}}'(x_\mathbf{d}) - 4^{n} g_{\mathbf{a} \mathbf{c}}'(x_{\mathbf{d}}) | \leq \sigma \right\} \leq  {\delta^{-1}} \sigma^{1/60}  .$$
\end{lemma}

\begin{proof}

Let $(\mathbf{b},\mathbf{c},\mathbf{d})$ be such that $|4^n g_{\mathbf{a} \mathbf{b}}'(x_\mathbf{d}) - 4^n g_{\mathbf{a} \mathbf{c}}'(x_{\mathbf{d}})| \leq \sigma$. Then:
$$ |S_{2n} \Phi \circ g_{\mathbf{a} \mathbf{b}}(x_{\mathbf{d}}) - S_{2n} \Phi \circ g_{\mathbf{a} \mathbf{c}}(x_{\mathbf{d}})| = \delta^{-1} \left| \ln\left( 4^n g_{\mathbf{a} \mathbf{b}}'(x_{\mathbf{d}}) \right) - \ln\left( 4^n g_{\mathbf{a} \mathbf{c}}'(x_{\mathbf{d}}) \right) \right| $$
$$ \leq \delta^{-1} e^{2 \delta n} \left|4^n g_{\mathbf{a} \mathbf{b}}'(x_{\mathbf{d}})  - 4^{n}  g_{\mathbf{a} \mathbf{c}}'(x_{\mathbf{d}})  \right| \leq \delta^{-1} e^{2 \delta n} \sigma.$$

We can then conclude using the previous lemma:
$$ 8^{-n} \# \left\{ (\mathbf{b}, \mathbf{c}, \mathbf{d}) \in \left(\{0,1\}^n\right)^3 , \ | 4^n g_{\mathbf{a} \mathbf{b}}'(x_\mathbf{d}) - 4^{n} g_{\mathbf{a} \mathbf{c}}'(x_{\mathbf{d}}) | \leq \sigma \right\} $$
$$ \leq 8^{-n} \# \left\{ (\mathbf{b}, \mathbf{c},\mathbf{d}) \in \left(\{0,1\}^n\right)^3 , \  |S_{2n} \Phi \circ g_{\mathbf{a} \mathbf{b}}(x_{\mathbf{d}}) - S_{2n} \Phi \circ g_{\mathbf{a} \mathbf{c}}(x_{\mathbf{d}}) | \leq \delta^{-1} \sigma e^{2 \delta n} \right\} $$
$$ \leq 2 \delta^{-1/2} (\delta^{-1} \sigma e^{2 \delta n})^{1/50} \leq \delta^{-1} \sigma^{1/60}  $$ \end{proof}

\begin{lemma}

We call a couple $(\mathbf{a},\mathbf{d}) \in (\{0,1\}^{n})^2$ regular if: 
$$\forall \sigma  \in [ e^{- 4 \varepsilon_0 n} , e^{-\varepsilon_0 \varepsilon_1 n/2} ], \quad 4^{-n} \# \{ \mathbf{b} , \mathbf{c} \in \{0,1\}^{n} , \ | 4^n g_{\mathbf{a} \mathbf{b}}'(x_\mathbf{d}) - 4^n g_{\mathbf{a} \mathbf{c}}'(x_\mathbf{d}) | \leq \sigma \} \leq \sigma^{1/100}.$$

Denote by $\mathcal{R}_{n}^2 \subset \{0,1\}^{(k+1)n}$ the set of regular couples. Then most couples are regular:

$$  4^{-n} \# \left( (\{0,1\}^{n})^2 \setminus \mathcal{R}_n^{2} \right) \lesssim  \delta^{-1} e^{- \varepsilon_0 \varepsilon_1 n/400} .$$

\end{lemma}

\begin{proof}

We use a dyadic decomposition: for each $\sigma \in [e^{-4 \varepsilon_0 n }, e^{-\varepsilon_0 \varepsilon_1 n/2 }]$, there exists $l \in \llbracket 0 , 4 \varepsilon_0 n \rrbracket$ such that $e^{-(l+1)} \leq \sigma \leq e^{-l}$. Hence
$$ \mathcal{R}_n^2 \subset \bigcap_{l \in \llbracket 0, \lfloor 4 \varepsilon_0 n \rfloor \rrbracket} \mathcal{R}_{n,l}^2 ,$$
where we denoted $$\mathcal{R}_{n,l}^2 := \left\{ (\mathbf{a},\mathbf{d}) \in (\{0,1\}^n)^2 \ \Big{|} \ 4^{-n} \# \{ \mathbf{b} , \mathbf{c} \in (\{0,1\}^{n})^2 , \ | 4^n g_{\mathbf{a} \mathbf{b}}'(x_\mathbf{d}) - 4^n g_{\mathbf{a} \mathbf{c}}'(x_\mathbf{d}) | \leq e^{-l} \} \leq e^{-(l+1)/100}  \right\}.$$
Markov's inequality and the previous lemma gives us the bound
$$ 4^{-n} \# \left( (\{0,1\}^{n})^2 \setminus \mathcal{R}_{n,l} \right) \leq e^{(l+1)/100} 16^{-n} \sum_{\mathbf{a},\mathbf{d}} \# \{ \mathbf{b} , \mathbf{c} \in (\{0,1\}^{n})^2 , \ | 4^n g_{\mathbf{a} \mathbf{b}}'(x_\mathbf{d}) - 4^n g_{\mathbf{a} \mathbf{c}}'(x_\mathbf{d}) | \leq e^{-l} \}$$ $$ \leq e^{(l+1)/100}  \left( \delta^{-1} e^{-l/60} \right) \lesssim  \delta^{-1} e^{-\varepsilon_0 \varepsilon_1 n/300}.$$
Which gives, by summing over $0 \leq l \leq 4 \varepsilon_0 n$, for $n$ large enough:
$$ 4^{-n} \# \left( (\{0,1\}^{n})^2 \setminus \mathcal{R}_n^{2} \right) \leq  \delta^{-1} e^{- \varepsilon_0 \varepsilon_1 n/400} .$$ \end{proof}

\begin{lemma}

Notice that a block $\mathbf{A}=\mathbf{a_0} \dots \mathbf{a}_k \in \{0,1\}^{(k+1)n}$ is regular if $(\mathbf{a}_{j-1},\mathbf{a}_j)$ is a regular couple, for all $j \in \llbracket 1,k\rrbracket$. Denote the set of regular blocks by $\mathcal{R}^{k+1}_n$. Then, most blocks are regular:

$$ 2^{-(k+1)n} \# \left( \{0,1\}^{(k+1)n} \setminus \mathcal{R}_n^{k+1} \right) \lesssim  \delta^{-1} e^{- \varepsilon_0 \varepsilon_1 n/400} .$$

\end{lemma}

\begin{proof}

The result follows from the previous lemma, noticing that $$ \mathcal{R}_n^{k+1} \subset \bigcap_{j=1}^k \{ \mathbf{A} \in \{0,1\}^{(k+1)n}, \ (\mathbf{a}_{j-1},\mathbf{a}_{j}) \in \mathcal{R}_n^2 \}. $$\end{proof}

\section{Acknowledgments}

Thanks are due to my PhD advisor, Frederic Naud, for a discussion on quasiconformal mappings and Brownian motion that lead to this paper. I would also like to thank Nguyen Viet Dang for suggesting that I find a way not to use Dolgopyat's estimates. Frederic Naud pointed out to me that an elementary argument exists in the work of Bourgain-Dyatlov, which led me to this explicit approach.  This work is part of the author's PhD and is funded by the Ecole Normale Superieure de Rennes.

\end{document}